\documentclass[graybox]{svmult}
\usepackage{mathptmx}       
\usepackage{helvet}         
\usepackage{courier}        
\usepackage{type1cm}        
\usepackage{makeidx}         
\usepackage{graphicx}        
\usepackage{multicol}        
\usepackage[bottom]{footmisc}
\usepackage{url}
\usepackage{amssymb}
\usepackage{amsmath}
\usepackage{mathtools}
\usepackage{algorithm}
\usepackage{algpseudocode}

\algnewcommand\algorithmicoptional{\textbf{Optional:}}
\algnewcommand\Optional{\item[\algorithmicoptional]}%
\algnewcommand\algorithmicinput{\textbf{Input:}}
\algnewcommand\Input{\item[\algorithmicinput]}%

\DeclareMathOperator{\CovHat}{\widehat{\Cov}}
\DeclareMathOperator{\Cov}{Cov}
\DeclareMathOperator{\Expectation}{\mathbb E} 
\DeclareMathOperator{\Hessian}{Hess}
\DeclareMathOperator{\Span}{Span}
\DeclareMathOperator{\Supp}{Supp}
\DeclareMathOperator{\Var}{Var}

\newcommand{\N}{\mathbb{N}}
\newcommand{\covat}[3]{\Cov_{#1}\left(#2,#3\right)}
\newcommand{\derivby}[1]{\frac{d}{d#1}}
\newcommand{\euler}{\mathrm{e}}
\newcommand{\expectat}[2]{{\Expectation}_{#1}\left[#2\right]}
\newcommand{\expof}[1]{\exp\left(#1\right)}
\newcommand{\logof}[1]{\log\left(#1\right)}
\newcommand{\naturals}{\N}
\newcommand{\normat}[2]{\left\Vert#2\right\Vert_{#1}}
\newcommand{\pdensities}{\mathcal P_>}
\newcommand{\reals}{\mathbb R}
\newcommand{\setof}[2]{\left\{#1 \colon #2 \right\}}
\newcommand{\set}[1]{\left\{#1\right\}}
\newcommand{\smodel}{\mathcal M}
\newcommand{\spanof}[1]{\Span\left(#1\right)}
\newcommand{\suppof}[1]{\Supp{#1}}
\newcommand{\upgamma}[2]{\Gamma\left(#1,#2\right)}
\newcommand{\varat}[2]{\Var_{#1}\left(#2\right)}

\makeindex             
\begin{document}

\title{Optimization via Information Geometry}
\titlerunning{Optimization via Information Geometry}

\author{Luigi Malag\`o and Giovanni Pistone}
\authorrunning{L. Malag\`o and G. Pistone}

\institute{Luigi Malag\`o \at Dipartimento di Informatica, Universit\`a degli Studi di Milano, Via Comelico, 39/41, 20135 Milano, Italy, \email{malago@di.unimi.it}
\and Giovanni Pistone \at de Castro Statistics, Collegio Carlo
Alberto, Via Real Collegio 30, 10024 Moncalieri, Italy, \email{giovanni.pistone@carloalberto.org}}

\maketitle

\abstract{Information Geometry has been used to inspire efficient algorithms for stochastic optimization, both in the combinatorial and the continuous case. We give an overview of the authors' research program and some specific contributions to the underlying theory.}

\abstract*{Information Geometry has been used to inspire efficient
  algorithms for black-box optimization, both in the combinatorial and
  in the continuous case. We give an overview of the authors' research program and some specific contribution to the underlying theory.}

\section{Introduction}
\label{sec:introduction}
The present paper is based on the talk given by the second author on May 21, 2013, to the Seventh International Workshop on Simulation in Rimini. Some pieces of research that were announced in that talk have been subsequently published \cite{malago|matteucci|pistone:2013CEC,pistone:2013Entropy,pistone:2013GSI}. Here we give a general overview, references to latest published results, and a number of specific topics that have not been published elsewhere.

 Let $(\Omega,\mathcal F,\mu)$ be a measure space, whose strictly positive probability densities form the algebraically open convex set $\pdensities$. An \emph{open statistical model} $(\smodel,\theta,B)$ is a parametrized subset of $\pdensities$, that is, $\smodel \subset \pdensities$ and $\theta \colon \smodel \to B$, where $\theta$ is a one-to-one mapping onto an open subset of a Banach space $B$. We assume in the following that $\Omega$ is endowed with a distance and $\mathcal F$ is its Borel $\sigma$-algebra. 

If $f \colon \Omega \to \reals$ is a bounded continuous function, the mapping $\smodel \ni p \mapsto \expectat p f$ is a \emph{Stochastic Relaxation} (SR) of $f$. The strict inequality $\expectat p f < \sup_{\omega \in \Omega} f(\omega)$ holds for all $p \in \smodel$, unless $f$ is constant. However, $\sup_{p \in \smodel} \expectat p f = \sup_{\omega \in \Omega} f(\omega)$  if there exist a probability measure $\nu$ in the weak closure of $\smodel\cdot\mu$ whose support is contained in the set of maximizing points of $f$, that is to say
\begin{equation*}
  \nu\setof{\omega \in \Omega}{f(\omega) = \sup_{\omega\in\Omega}f(\omega)} = 1, \quad \text{or} \quad \int f \ d\nu = \sup_{\omega \in \Omega} f(\omega).
\end{equation*}
Such a $\nu$ belongs to the border of $\smodel\cdot\mu$. For a
discussion of the border issue for finite $\Omega$, see 
\cite{malago|pistone:arXiv1012.0637}. Other relaxation methods have
been considered, e.g., \cite{arnoldetal:2011arXiv,wierstra|Schaul|peters|schmidhuber:2008}.

A \emph{SR optimization method} is an algorithm producing a sequence $p_n \in \smodel$, $n \in
\naturals$, which is expected to converge to the probability measure
$\nu$, so that $\lim_{n \to \infty} \expectat {p_n} f =
\sup_{\omega \in \Omega} f(\omega)$.  Such algorithms are best studied in the framework of \emph{Information Geometry} (IG), that is, the differential geometry of statistical models. See \cite{amari|nagaoka:2000} for a general treatment of IG and   
\cite{arnoldetal:2011arXiv,bensadon:2013arXiv1309.7168,malago:2012thesis,malago|matteucci|pistone:2009NIPS,malago|matteucci|pistone:2011CEC,malago|matteucci|pistone:2011FOGA,malago|matteucci|pistone:2013CEC} for applications to SR. All the
quoted literature refers to the case where the model Banach space of
the statistical manifold, i.e., the parameter space, is finite
dimensional, $B = \reals^d$. An infinite dimensional version of IG has
been developed, see \cite{pistone:2013GSI} for a recent presentation
together with new results, and references therein for a detailed bibliography.
The nonparametric version is unavoidable in applications to evolution
equations in Physics \cite{pistone:2013Entropy}, and it is useful even
when the sample space is finite
\cite{malago|pistone:inprogressEntropy}.
%
\section{Stochastic relaxation on an exponential family}
\label{sec:exponential}%
We recall some basic facts on exponential families, see \cite{brown:86}.
\begin{enumerate}
\item The exponential family $q_\theta=\expof{\sum_{j=1}^d \theta_jT_j-\psi(\theta)}\cdot p$, $\expectat p {T_j} = 0$, is a statistical model $\smodel = \set{q_\theta}$ with parametrization $q_\theta \mapsto \theta \in \reals^d$.
  \item $\psi(\theta) = \logof{\expectat p {\euler^{\theta \cdot T}}}$, $\theta \in \reals^d$, is convex and lower semi-continuous.
  \item $\psi$ is analytic on the (non empty) interior $\mathcal U$ of its proper domain.
  \item $\nabla \psi(\theta) = \expectat \theta {T}$, $T = (T_1,\dots,T_d)$.
\item $\Hessian \psi(\theta) = \varat \theta T$.
  \item $\mathcal U \ni \theta \mapsto \nabla \psi(\theta) = \eta \in
    \mathcal N$ is one-to-one, analytic, and monotone; $\mathcal N$ is the interior of the \emph{marginal polytope}, i.e., the convex set generated by $\setof{T(\omega)}{ \omega \in \Omega}$.
\item The gradient of the SR of $f$ is
\begin{equation*}
\nabla(\theta \mapsto \expectat \theta f) = (\covat \theta f {T_1},\dots, \covat \theta f {T_d}),
\end{equation*}
which suggests to take the least squares approximation of $f$ on
$\spanof{T_1,\dots,T_d}$ as direction of \emph{steepest ascent}, see \cite{malago|matteucci|pistone:2011FOGA}.
\item The representation of the gradient in the scalar product with respect to  $\theta$ is called \emph{natural gradient}, see \cite{amari|nagaoka:2000,amari:1998natural,malago|pistone:inprogressEntropy}.
\end{enumerate}

Different methods can be employed to generate a maximizing sequence
of densities $p_n$ is a statistical model $\mathcal M$. A first
example is given by Estimation of Distribution
Algorithms (EDAs)~\cite{larranaga|lozano:2001}, a large family of iterative algorithms where the
parameters of a density are estimated after sampling and
selection, in order to favor samples with larger values for $f$, see
Example~\ref{eda}. Another approach is to evaluate the gradient of
$\expectat {p} f$ and follow the direction of the natural gradient 
over $\mathcal M$, as illustrated in Example~\ref{sngd}. 

\begin{example}[EDA from {\cite{malago|matteucci|pistone:2013CEC}}]
\label{eda}
An \emph{Estimation of Distribution Algorithm} is a SR optimization
algorithm based on sampling, selection and estimation, see \cite{larranaga|lozano:2001}.
\begin{quotation}
\begin{algorithmic}
\Input $N,M$ \Comment population size, selected population size
\Input $\mathcal M=\{p(x;\xi)\}$ \Comment parametric model
\State $t \gets 0$
\State $\mathcal P^t = \Call{InitRandom}$ \Comment random initial population
\Repeat
\State $\mathcal P^t_s = \Call{Selection}{\mathcal P^t,M}$ \Comment select $M$
samples
\State $\xi^{t+1} = \Call{Estimation}{\mathcal P^t_s,\mathcal M}$
\Comment opt.\ model selection
\State$\mathcal P^{t+1} = \Call{Sampler}{\xi^{t+1},N}$ \Comment $N$ samples
 \State $t \gets t+1$
\Until{$\Call{StoppingCriteria}$}
\end{algorithmic}
\end{quotation}
\end{example}

\begin{example}[SNGD from {\cite{malago|matteucci|pistone:2013CEC}}]
\label{sngd}
\emph{Stochastic Natural Gradient Descent}~\cite{malago|matteucci|pistone:2011FOGA} is a SR algorithm that
  requires the estimation of the gradient.
\begin{quotation}
\begin{algorithmic}
\Input $N, \lambda$ \Comment population size, learning rate
\Optional $M$ \Comment selected population size (default $M=N$)
\State $t \gets 0$
\State $\theta^t \gets (0, \dots, 0)$ \Comment uniform distribution
\State $\mathcal P^t \gets \Call{InitRandom}$ \Comment random initial population
\Repeat
\State $\mathcal P^t_s = \Call{Selection}{\mathcal P^t,M}$ \Comment opt.\ select $M$
samples
\State $\widehat \nabla \mathbb E [f] \gets \CovHat(f,T_i)_{i=1}^d$ \Comment empirical covariances
\State $\widehat I \gets [\CovHat(T_i,T_j)]_{i,j=1}^d$
\Comment $\{T_i(x)\}$ may be learned 
\State $\theta^{t+1} \gets \theta^t - \lambda \widehat I^{-1} \widehat
\nabla \mathbb E [f]$
\State$\mathcal P^{t+1} \gets \Call{GibbsSampler}{\theta^{t+1},N}$ \Comment $N$ samples
 \State $t \gets t+1$
\Until{$\Call{StoppingCriteria}$}
\end{algorithmic}
\end{quotation}
\end{example}

Finally, other algorithms are based on Bregman
divergence. Example~\ref{binomial} illustrates the connection with the
exponential family.

\begin{example}[Binomial B$(n,p)$]\label{ex:bin1}
\label{binomial}
On the finite sample space $\Omega = \set{0,\dots,n}$ with $\mu(x) = \binom n x$, consider the exponential family $p(x;\theta) = \expof{\theta x - n \logof{1+\euler^\theta}}$. With respect to the expectation parameter $\eta = n\euler^\theta/(1+\euler^\theta) \in ]0,n[$ we have $p(x;\eta)= (\eta/ n)^x (1 - \eta / n)^{n-x}$, which is the standard presentation of the binomial density.

The standard presentation is defined for $\eta = 0,n$, where the exponential formula is not. In fact, the conjugate $\psi_*(\eta)$ of $\psi(\theta) = n \logof{1+\euler^\theta}$ is
\begin{equation*}
  \psi_*(\eta) =
  \begin{cases}
    +\infty & \text{if $\eta < 0$ or $\eta > n$,} \\
    0 & \text{if $\eta = 0, n$,} \\
     \eta \logof{\frac{\eta}{n-\eta}} - n \logof{\frac{n}{n-\eta}} & \text{if $ 0 < \eta < n$.} 
  \end{cases}
\end{equation*}

We have
\begin{align*}
 \log p(x;\eta) &= \logof{\frac{\eta}{n-\eta}} (x - \eta) + \psi_*(\eta), \quad \eta \in ]0,n[ \\
&= \psi_*'(\eta)(x-\eta) + \psi_*(\eta) \le \psi_*(x).
\end{align*}
For $x \ne 0,n$, the sign of  $\psi_*'(\eta)(x-\eta)$ is eventually negative as $\eta \to 0,n$, hence
\begin{equation*}
  \lim_{\eta \to 0,n} \log p(x;\eta)  = \lim_{\eta \to 0,n} \psi_*'(\eta)(x-\eta) + \psi_*(\eta) = -\infty.
\end{equation*}
If $x = 0,n$, the sign of both $\psi_*'(\eta)(0-\eta)$ and
$\psi_*'(\eta)(n-\eta)$ is eventually positive as $\eta \to 0$ and
$\eta \to n$, respectively. The limit is bounded by $0 = \psi_*(x)$,
for $x = 0,n$.

The argument above is actually general. It has been observed by \cite{benerjee|merugu|dhillon|ghosh:2005} that the Bregman divergence $D_{\psi_*}(x \Vert \eta) = \psi_*(x) -  \psi_*(\eta) -  \psi_*'(\eta)(x-\eta) \ge 0$ provides an interesting form of the density as $ p(x;\eta) = \euler^{- D_{\psi_*}(x \Vert \eta)} \euler^{\psi_*(x)} \propto \euler^{- D_{\psi_*}(x \Vert \eta)}$.
%

\end{example}

\section{Exponential manifold}
\label{sec:exponential-manifold}
The set of positive probability densities $\pdensities$ is a convex
subset of $L^1(\mu)$. Given a $p \in \pdensities$, every $q \in
\pdensities$ can be written as $q = \euler^v \cdot p$ where $v =
\logof{\frac qp}$. Below we summarize, together with a few new
details, results from
\cite{pistone:2013Entropy,pistone:2013GSI} and references therein, and the unpublished \cite{SST:inprogress}.

\begin{definition}[Orlicz $\Phi$-space {\cite{krasnoselskii|rutickii:61}, \cite[Chapter II]{musielak:1983}, \cite{rao|ren:2002}}]
Define $\phi(y) = \cosh y - 1$. The Orlicz $\Phi$-space $L^\Phi(p)$ is the vector space of all random variables such that $\expectat p {\Phi(\alpha u)}$ is finite for some $\alpha > 0$. Equivalently, it is the set of all random variables $u$ whose Laplace transform under $p \cdot \mu$, $t \mapsto \hat u_p(t) = \expectat p {\euler^{t u}}$ is finite in a neighborhood of 0. We denote by $M^\Phi(p) \subset L^\Phi(p)$ the vector space of random variables whose Laplace transform is always finite.
\end{definition}

\begin{proposition}[Properties of the $\Phi$-space]\ 
  \begin{enumerate}
  \item The set $S_{\le1} = \setof{u \in L^\Phi(p)}{\expectat p {\Phi(u)} \le 1}$ is the closed unit ball of the complete norm
    \begin{equation*}
      \normat p u = \inf\setof{\rho > 0}{\expectat p {\Phi\left(\frac u\rho \right)} \le 1}
    \end{equation*}
on the $\Phi$-space. For all $a \ge 1$ the continuous injections $L^\infty(\mu) \hookrightarrow L^\Phi(p) \hookrightarrow L^a(p)$ hold.
  \item $\normat p u = 1$ if either $\expectat p {\Phi(u)} = 1$ or $\expectat p {\Phi(u)} < 1$ and $\expectat p {\Phi\left(\frac u\rho\right)} = \infty$ for $\rho > 1$. If $\normat p u > 1$ then $\normat p u \le \expectat p {\Phi(u)}$. In particular, $\lim_{\normat p u \to \infty} \expectat p {\Phi\left(u\right)} = \infty$.
  \item $M^\Phi(p)$ is a closed and separable subspace of $L^{\Phi}(p)$.
\item $L^\Phi(p) = L^\Phi(q)$ as Banach spaces if, and only if, $\int p^{1-\theta} q^\theta\ d\mu$ is finite on a neighborhood of $[0,1]$.
  \end{enumerate}
  \end{proposition}
  \begin{proof}\ 
    \begin{enumerate}
    \item See \cite{krasnoselskii|rutickii:61}, \cite[Chapter II]{musielak:1983}, \cite{rao|ren:2002}.
    \item The function $\reals_\ge \ni \alpha \mapsto \hat u(t) = \expectat p {\Phi(\alpha u)}$ is increasing, convex, lower semi-continuous. If for some $t_+ > 1$ the value $\hat u(t_+)$ is finite, we are in the first case and $\hat u(1) = 1$. Otherwise, we have $\hat u(1) \le 1$. If $\normat p u > a > 1$, so that $\normat p {\frac a{\normat p u} u} > 1$, hence
      \begin{equation*}
        1 < \expectat p {\Phi\left(\frac a{\normat p u} u\right)} \le \frac a{\normat p u} \expectat p {\Phi\left(u\right)},
      \end{equation*}
and $\normat p u < a \expectat p {\Phi\left(u\right)}$, for all $a > 1$.
\item See \cite{krasnoselskii|rutickii:61}, \cite[Chapter II]{musielak:1983}, \cite{rao|ren:2002}.
\item See \cite{cena|pistone:2007,SST:inprogress}.   
\end{enumerate}
  \end{proof}

\begin{example}[Boolean state space]
In the case of a finite state space, the moment generating function is
finite everywhere, but its computation can be challenging. We discuss
in particular the Boolean case $\Omega=\set{+1,-1}^n$ with counting
reference measure $\mu$ and uniform density $p(x)=2^{-n}$, $x\in
\Omega$. In this case there is a huge literature from statistical
physics, e.g., \cite[Ch. VII]{gallavotti:1999short}. A generic real
function on $\Omega$---called 
pseudo-Boolean \cite{boros|hammer:2002} in the combinatorial optimization literature---has the form
$u(x)=\sum_{\alpha \in L} \hat u(\alpha) x^\alpha$, with $L =
\set{0,1}^n$, $x^\alpha = \prod_{i=1}^n x_i^{\alpha_i}$, $\hat
u(\alpha) = 2^{-n} \sum_{x \in \Omega} u(x) x^\alpha$.  

As $\euler^{a x} = \cosh(a) + \sinh(a) x$ if $x^2 = 1$ i.e., $x = \pm 1$, we have
 \begin{align*}
  \euler^{tu(x)} &= \expof{\sum_{\alpha \in \suppof{\hat{u}}} t\hat{u}(\alpha) x^\alpha} = \prod_{\alpha \in \suppof{\hat{u}}} \euler^{t\hat{u}(\alpha) x^\alpha} \\
&= \prod_{\alpha \in \suppof{\hat{u}}} \left(\cosh(t\hat{u}(\alpha)) + \sinh(t\hat{u}(\alpha)) x^\alpha \right) \\&= \sum_{B \subset \suppof{\hat u}}\prod_{\alpha \in B^c} \cosh(t\hat{u}(\alpha)) \prod_{\alpha \in B} \sinh(t\hat{u}(\alpha)) x^{\sum_{\alpha \in B}\alpha}.
\end{align*}

The moment generating function of $u$ under the uniform density $p$ is
\begin{equation*}
  t \mapsto \sum_{B \in \mathcal B(\hat u)} \prod_{\alpha \in B^c} \cosh(t\hat u(\alpha)) \prod_{\alpha \in B} \sinh(t\hat u(\alpha)),
\end{equation*}
where $\mathcal B(\hat u)$ are those $B \subset \suppof{\hat u}$ such that $\sum_{\alpha \in B} \alpha = 0 \mod 2$. We have
\begin{equation*}
  \expectat p \Phi(tu) =  \sum_{B \in \mathcal B_0(\hat u)} \prod_{\alpha \in B^c} \cosh(t\hat u(\alpha)) \prod_{\alpha \in B} \sinh(t\hat u(\alpha)) - 1,
\end{equation*}
where $\mathcal B_0(\hat u)$ are those $B \subset \suppof{\hat u}$ such that $\sum_{\alpha \in B} \alpha = 0 \mod 2$ and $\sum_{\alpha \in \suppof{\hat u}} \alpha = 0$. 

If $S$ is the $\set{1,\dots,n}\times\suppof{\hat u}$ matrix with elements $\alpha_i$ we want to solve the system $Sb = 0 \mod 2$ to find all elements of $\mathcal B$; we add the equation $\sum b = 0 \mod 2$ to find $\mathcal B_0$. The simplest example is $u(x) = \sum_{i=1}^n c_i x_i$,
\end{example}

\begin{example}[The sphere is not smooth in general]
We look for the moment generating function of the density
\begin{equation*}
  p(x) \propto (a+x)^{-\frac32} \euler^{-x}, \quad x >0,
\end{equation*}
where $a$ is a positive constant. From the incomplete gamma integral
\begin{equation*}
\upgamma{-\frac12}x = \int_x^\infty s^{-\frac12-1} \euler^{-s}\ ds, \quad x > 0,
\end{equation*}
we have for $\theta, a > 0$,
\begin{equation*}
\derivby x \upgamma{-\frac12}{\theta(a+x)} = -\theta^{-\frac12}\euler^{-\theta a}(a+x)^{-\frac32} \euler^{-\theta x}.
\end{equation*}

We have, for $\theta \in \reals$,
\begin{equation*}
 C(\theta,a) = \int_0^\infty (a+x)^{-\frac32} \euler^{-\theta x}\ dx =
 \begin{cases}
   \sqrt\theta \euler^{\theta a} \upgamma{-\frac12}{\theta a} & \text{if $\theta > 0$}. \\
\frac1{2\sqrt a} & \text{if $\theta = 0$}, \\ +\infty & \text{if $\theta < 0$}. 
 \end{cases}
\end{equation*}
or, $C(\theta,a) = \frac12 a^{-\frac12} - \frac{\sqrt \pi \theta}2 \euler^{\theta a} R_{1/2,1}(\theta a)$ if $\theta \le 1$, $+\infty$ otherwise, where $R_{1/2,1}$ is the survival function of the Gamma distribution with shape $1/2$ and scale 1.

The density $p$ is obtained with $\theta=1$,
\begin{equation*}
  p(x) = C(1,a)^{-1} (a+x)^{-\frac32} \euler^{-x} = \frac{(a+x)^{-\frac32} \euler^{-x}}{\euler^{a} \upgamma{-\frac12}{a}}, \quad x > o,
\end{equation*}
and, for the random variable $u(x) = x$, the function 
\begin{align*}
 \alpha \mapsto \expectat p {\Phi(\alpha u)} &= \frac1{\euler^{a} \upgamma{-\frac12}{a}}\int_0^\infty  (a+x)^{-\frac32}\frac{\euler^{-(1-\alpha)x}+\euler^{-(1+\alpha)x}}2\ dx - 1 \notag \\
&= \frac{C(1-\alpha,a)+C(1+\alpha,a)}{2C(1,a)} - 1 \label{eq:nonsteep}
\end{align*}
is convex lower semi-continuous on $\alpha \in \reals$, finite for $\alpha \in [-1,1]$, infinite otherwise, hence not steep. Its value at $\alpha = 1$ is
\begin{align*}
  \expectat p {\Phi(u)} &= \frac1{\euler^{a} \upgamma{-\frac12}{a}}\int_0^\infty  (a+x)^{-\frac32}\frac{1 + \euler^{-2x}}2 \ dx - 1 \\
&= \frac{C(0,a)+C(2,a)}{2C(1,a)} - 1
\end{align*}
\end{example}

\begin{example}[Normal density] Let $p(x) = (2\pi)^{-1/2} \euler^{-(1/2)x^2}$. Consider a generic quadratic polynomial $u(x) = a + bx + \frac 12 cx^2$. We have for $tc \ne 1$
\begin{equation*}
  t(a + bx + \frac12 cx^2) - \frac12 x^2 == - \frac1{2(1-tc)^{-1}} \left(x - \frac{tb}{1-tc}\right)^2 +\frac12\frac{t^2b^2-2ta(1-tc)}{(1-tc)},
\end{equation*}
hence
\begin{equation*}
  \expectat p {\euler^{tu}} =
  \begin{cases}
    +\infty & \text{if $tc \le 1$,} \\ \sqrt{1-tc} \expof{\dfrac12\dfrac{t^2b^2-2ta(1-tc)}{(1-tc)}} & \text{if $tc < 1$}.
  \end{cases}
\end{equation*}

If, and only if, $-1 < c < 1$, we have
\begin{multline*}
  \expectat p {\Phi(u)} = \\ \frac12 \sqrt{1-c} \expof{\dfrac12\dfrac{b^2-a(1-c)}{(1-c)}} + \frac12 \sqrt{1+c} \expof{\dfrac12\dfrac{b^2-a(1+c)}{(1+c)}} - 1.
\end{multline*}
\end{example}

\section{Vector bundles}
\label{sec:vb}

Vector bundles are constructed as sets of couples $(p,v)$ with $p \in \pdensities$ and $v$ is some space of random variables such that $\expectat p v = 0$. The tangent bundle is obtained when the vector space is $L^{\Phi}_0(p)$. The Hilbert bundle is defined as $H\pdensities = \setof{(p,v)}{p \in \pdensities, v \in L_0^2(p)}$. We refer to \cite{pistone:2013Entropy} and \cite{malago|pistone:inprogressEntropy} were charts and affine connections on the Hilbert bundle are derived from the isometric transport
\begin{equation*}
  L_0^2(p) \ni u \mapsto \sqrt{\frac pq} u  -  \left(1 +  \expectat q {\sqrt{\frac pq}}\right)^{-1} \left(1 + \sqrt{\frac pq}\right) \expectat q {\sqrt{\frac pq}u} \in L_0^2(q).
\end{equation*}
In turn, an isometric trasport $U_p^q \colon L_0^2(p) \to L_0^2(q)$ can be used to compute the derivative of a vector field in the Hilbert bundle, for example the derivative of the gradient of a relaxed function.
  
The resulting second order structure is instrumental in computing the
Hessian of the natural gradient of the SR function. This allows the
design a second order approximation method, as it is suggested in
\cite{absil|mahony|sepulchre:2008} for general Riemannian manifolds,
and applied to SR in \cite{malago|pistone:inprogressEntropy}. A second order structure is also used to define the curvature of a
statistical manifold and, possibly, to compute its geodesics, see
\cite{bensadon:2013arXiv1309.7168}  for applications to optimization.

\bibliographystyle{spmpsci}

\begin{thebibliography}{10}
\providecommand{\url}[1]{{#1}}
\providecommand{\urlprefix}{URL }
\expandafter\ifx\csname urlstyle\endcsname\relax
  \providecommand{\doi}[1]{DOI~\discretionary{}{}{}#1}\else
  \providecommand{\doi}{DOI~\discretionary{}{}{}\begingroup
  \urlstyle{rm}\Url}\fi

\bibitem{absil|mahony|sepulchre:2008}
Absil, P.A., Mahony, R., Sepulchre, R.: Optimization algorithms on matrix
  manifolds.
\newblock Princeton University Press, Princeton, NJ (2008).
\newblock With a foreword by Paul Van Dooren

\bibitem{amari|nagaoka:2000}
Amari, S., Nagaoka, H.: Methods of information geometry.
\newblock American Mathematical Society, Providence, RI (2000).
\newblock Translated from the 1993 Japanese original by Daishi Harada

\bibitem{amari:1998natural}
Amari, S.I.: Natural gradient works efficiently in learning.
\newblock Neural Computation \textbf{10}(2), 251--276 (1998)

\bibitem{arnoldetal:2011arXiv}
{Arnold}, L., {Auger}, A., {Hansen}, N., {Ollivier}, Y.: {Information-Geometric
  Optimization Algorithms: A Unifying Picture via Invariance Principles}
  (2011v1; 2013v2).
\newblock ArXiv:1106.3708

\bibitem{benerjee|merugu|dhillon|ghosh:2005}
Banerjee, A., Merugu, S., Dhillon, I.S., Ghosh, J.: Clustering with bregman
  divergences.
\newblock Journal of Machine Learning Research \textbf{6}, 1705--1749 (2005)

\bibitem{bensadon:2013arXiv1309.7168}
Bensadon, J.: Black-box optimization using geodesics in statistical manifolds.
\newblock ArXiv:1309.7168

\bibitem{boros|hammer:2002}
Boros, E., Hammer, P.L.: Pseudo-{B}oolean optimization.
\newblock Discrete Appl. Math. \textbf{123}(1-3), 155--225 (2002).
\newblock Workshop on Discrete Optimization, DO'99 (Piscataway, NJ)

\bibitem{brown:86}
Brown, L.D.: Fundamentals of statistical exponential families with applications
  in statistical decision theory.
\newblock No.~9 in IMS Lecture Notes. Monograph Series. Institute of
  Mathematical Statistics (1986)

\bibitem{cena|pistone:2007}
Cena, A., Pistone, G.: Exponential statistical manifold.
\newblock Ann. Inst. Statist. Math. \textbf{59}(1), 27--56 (2007)

\bibitem{gallavotti:1999short}
Gallavotti, G.: Statistical mechanics: A short treatise.
\newblock Texts and Monographs in Physics. Springer-Verlag, Berlin (1999)

\bibitem{krasnoselskii|rutickii:61}
Krasnosel'skii, M.A., Rutickii, Y.B.: Convex Functions and {O}rlicz Spaces.
\newblock Noordhoff, Groningen (1961).
\newblock Russian original: (1958) Fizmatgiz, Moskva

\bibitem{larranaga|lozano:2001}
Larra\~naga, P., Lozano, J.A. (eds.): Estimation of Distribution Algoritms. A
  New Tool for evolutionary Computation.
\newblock No.~2 in Genetic Algorithms and Evolutionary Computation. Springer
  (2001)

\bibitem{malago:2012thesis}
Malag\`o, L.: On the geometry of optimization based on the exponential family
  relaxation.
\newblock Ph.D. thesis, Politecnico di Milano (2012)

\bibitem{malago|matteucci|pistone:2009NIPS}
Malag\`o, L., Matteucci, M., Pistone, G.: Stochastic relaxation as a unifying
  approach in 0/1 programming (2009).
\newblock NIPS 2009 Workshop on Discrete Optimization in Machine Learning:
  Submodularity, Sparsity \& Polyhedra (DISCML), Dec 11 2009, Whistler, Canada

\bibitem{malago|matteucci|pistone:2011CEC}
Malag\`o, L., Matteucci, M., Pistone, G.: Stochastic natural gradient descent
  by estimation of empirical covariances.
\newblock In: Proc. of IEEE CEC, pp. 949--956 (2011)

\bibitem{malago|matteucci|pistone:2011FOGA}
Malag\`{o}, L., Matteucci, M., Pistone, G.: Towards the geometry of estimation
  of distribution algorithms based on the exponential family.
\newblock In: Proceedings of the 11th workshop on Foundations of genetic
  algorithms, FOGA '11, pp. 230--242. ACM, New York, NY, USA (2011)

\bibitem{malago|matteucci|pistone:2013CEC}
Malag\`o, L., Matteucci, M., Pistone, G.: Natural gradient, fitness modelling
  and model selection: A unifying perspective.
\newblock In: Proc. of IEEE CEC, pp. 486--493 (2013)

\bibitem{malago|pistone:arXiv1012.0637}
Malag\`o, L., Pistone, G.: A note on the border of an exponential family
  (2010).
\newblock ArXiv:1012.0637v1

\bibitem{malago|pistone:inprogressEntropy}
Malag\`o, L., Pistone, G.: Combinatorial optimization with information
  geometry: Newton method (2013).
\newblock In progress

\bibitem{musielak:1983}
Musielak, J.: Orlicz spaces and modular spaces, \emph{Lecture Notes in
  Mathematics}, vol. 1034.
\newblock Springer-Verlag, Berlin (1983)

\bibitem{pistone:2013Entropy}
Pistone, G.: Examples of application of nonparametric information geometry to
  statistical physics.
\newblock Entropy \textbf{15}(10), 4042--4065 (2013)

\bibitem{pistone:2013GSI}
Pistone, G.: Nonparametric information geometry.
\newblock In: F.~Nielsen, F.~Barbaresco (eds.) Geometric Science of
  Information, no. 8085 in LNCS, pp. 5--36. Springer-Verlag, Berlin Heidelberg
  (2013).
\newblock GSI 2013 Paris, France, August 28-30, 2013 Proceedings

\bibitem{rao|ren:2002}
Rao, M.M., Ren, Z.D.: Applications of {O}rlicz spaces, \emph{Monographs and
  Textbooks in Pure and Applied Mathematics}, vol. 250.
\newblock Marcel Dekker Inc., New York (2002)

\bibitem{SST:inprogress}
Santacroce, M., Siri, P., Trivellato, B.: New results on mixture and
  exponential models by {O}rlicz spaces (2013).
\newblock In progress

\bibitem{wierstra|Schaul|peters|schmidhuber:2008}
Wierstra, D., Schaul, T., Peters, J., Schmidhuber, J.: Natural evolution
  strategies.
\newblock In: Proc. of IEEE CEC, pp. 3381--3387 (2008)

\end{thebibliography}

\end{document}